\documentclass[12pt]{article}
\usepackage[margin=1in]{geometry}
\usepackage{amsthm}
\usepackage{amsmath}
\usepackage{amssymb}
\usepackage{hyperref}
\usepackage{xcolor}
\usepackage{tkz-euclide}
\usepackage[hang,flushmargin]{footmisc}
\usetikzlibrary{graphs}
\usetikzlibrary{chains}
\usetikzlibrary{matrix}

\hypersetup{
    colorlinks,
    linkcolor={black},
    citecolor={black},
    urlcolor={black}
}

\newcommand{\cgame}{\langle C,h\rangle} 
\newcommand{\NN}{\mathbb{N}} 

\newtheorem{theo}{Theorem}[section]
\newtheorem{lemma}[theo]{Lemma}
\newtheorem{defi}[theo]{Definition}

\title{The Hat Guessing Number of Cactus Graphs and Cycles }
\author{\normalsize
Jeremy Chizewer\thanks{Corresponding Author.
Department of Combinatorics and Optimization, University of Waterloo,\hspace{2em}
Email: {\tt jchizewer@uwaterloo.ca}.
}\and\normalsize
I.M.J. McInnis
\thanks{Department of Mathematics, Princeton University, Email: {\tt imj.mcinnis@gmail.com}}
\and\normalsize
Mehrdad Sohrabi
\thanks{
Department of Combinatorics and Optimzation and Department of Computer Science, University of Waterloo, Email: {\tt msohrabi@uwaterloo.ca}.}
\and\normalsize
Shriya Kaistha 
\thanks
{Department of Computer Science, University of Waterloo,
Email: {\tt shriya.kaistha@uwaterloo.ca}}}

\begin{document}

\date{}
\maketitle
\begin{abstract}
We study the hat guessing game on graphs. In this game, a player is placed on 
each vertex $v$ of a graph $G$ and assigned a colored hat from $h(v)$
possible colors. Each player makes a deterministic guess on their hat color 
based on the colors assigned to the players on neighboring vertices, and 
the players win if at least one player correctly guesses his assigned color.
If there exists a strategy that ensures at least one player guesses
correctly for every possible assignment of colors, the game defined by
$\langle G,h\rangle$ is called winning. The hat guessing number of $G$ is the largest 
integer $q$ so that if $h(v)=q$ for all $v\in G$ then $\langle G,h\rangle$ is winning.

In this note, we determine whether $\langle G,h\rangle $ is winning for any $h$ whenever 
$G$ is a cycle, resolving a conjecture of Kokhas and Latyshev in the affirmative, and extending it. 
We then use this result to determine the hat guessing number of every cactus 
graph, graphs in which every pair of cycles share at most one vertex. 
%
\end{abstract}

\section{Introduction}
Given a simple undirected graph $G$ and a function $h:V(G)\to\NN$ mapping the vertices
of $G$ to the positive integers, the \emph{hat guessing game} denoted
$\langle G,h\rangle$ is a puzzle in which for every vertex $v\in G$ a player is placed and  
assigned a hat from $h(v)$ possible colors. Each player attempts to guess the color
of his hat according to a predetermined strategy that takes as input the colors assigned to 
players at neighboring vertices. The players are a team and are said to \emph{win} 
for a given coloring if at least one player guesses correctly. The hat guessing game 
$\langle G,h\rangle$ is said to be \emph{winning} if there exists some set of strategies for the 
players so that for every possible color assignment, the players win. The players \emph{lose} on
a coloring if no player guesses correctly; if there exists such a coloring for every possible 
set of guessing strategies, the game is called \emph{losing}.

We refer to the function $h$ as the \emph{hatness} function, and for a vertex $v\in G$ we call the 
value $h(v)$ the hatness of $v$. A \emph{subgame} $\langle G',h'\rangle$ of a hat guessing game 
$\langle G,h\rangle$ is a game where
$G'\subseteq G$ is a subgraph, and the hatness function $h' = h|_{v\in G'}$ is the hatness 
function $h$ restricted to the vertices of $G$ in $G'$. Moreover, we call a subgame \emph{proper} if
$G'$ is a proper subgraph of $G$. The \emph{hat guessing number} of a graph, 
$G$, denoted $HG(G)$, is the largest integer $q$ such that $\langle G, \star q\rangle$ is a winning 
game, where $\star q$ denotes the constant function $h(\cdot) =q$. 
We call a vertex $v$ \emph{deletable} if the subgame on $G\backslash\{v\}$ is winning. 
Note that if a game has a winning subgame then it is winning,
and if a game is winning then decreasing the hatness of any vertex results in a winning game.  

The hat guessing number was introduced in \cite{BHKL}.
Previous results have attempted to bound the hat guessing number of graphs, but few give 
exact answers. Some exact results include the hat guessing number of trees (folklore), 
cycles~\cite{WS}, pseudotrees~\cite{KL18}, complete graphs (folklore), and some windmills 
and book graphs~\cite{HIP}. There are also upper and lower bounds (that are somewhat far apart) 
on complete bipartite graphs~\cite{BHKL, GG, ABST},
upper bounds on outerplanar graphs~\cite{PB, KMS}, and 
lower bounds on planar graphs~\cite{AC, KL, LK}.

In this note, we study the hat guessing problem for cycles and cactus graphs. 
Our first result resolves~\cite[Conjecture 5.2]{KL21} affirmatively and goes further, showing exactly which games on cycles are winning. The second result adds 
cactus graphs to the growing list of graph classes for which the hat guessing number is known 
exactly.

\begin{theo}\label{theo:cycles}
Let $C$ be a cycle on $n$ vertices such that the hatness $h$ of each 
vertex $v$ in $C$ satisfies $h(v)\geq 2$. Then $\cgame$ is winning 
if and only if $C$ satisfies at least one of the following conditions:
\begin{enumerate}
\item The length is $4$ or divisible by $3$ and $h(v) \leq 3$ for all $v\in C$.\label{szc}
\item The length is $3$ and $\sum_{v\in C} \frac{1}{h(v)} \geq 1$.\label{k3}
\item The game $\cgame$ contains a winning proper subgame.\label{subgame}
\item The hatnesses satisfy both of the following properties:\label{props}
\begin{enumerate}
\item There exists a sequence of adjacent
vertices in $C$ with hatnesses $(2,3,3)$ or $(3,2,3)$.\label{seq}
\item For all $v\in C$, $h(v)\leq 4$, \label{atmost4}
\end{enumerate}
\end{enumerate} 
\end{theo}

The forward implication for Condition~\ref{szc} is \cite[Theorem 1]{WS}; 
for Condition~\ref{k3}, it is \cite[Theorem 2.1]{KL21}; for
Condition~\ref{subgame}, it is trivial; and for Condition~\ref{props}, it is \cite[Theorem 5.1]{KL21}. 
We prove the reverse implications---that $\langle C,h\rangle$ is losing if 
none of the above conditions hold---in Section~\ref{sec:cycles}.

Our next result determines the hat guessing number of cactus graphs. 
A \emph{cactus graph} is a connected graph in which every pair of cycles share at
most one vertex. Figure~\ref{fig:glue-cactus} shows an example of a cactus graph (on the left). 
A useful subset of the class of cactus graphs is the class of \emph{pseudotrees}, which are 
connected graphs with at most one cycle. 
The hat guessing number of cactus graphs was first studied in \cite{PB},\footnote{The 
bound of $16$ was mentioned in the preprint 
(\href{https://arxiv.org/abs/2109.13422}{arXiv:2109.13422}) 
but was removed from the publication} 
which derived a bound of $16$ for any cactus graph. 
We apply Theorem~\ref{theo:cycles} to exactly determine the 
hat guessing number of every cactus graph.

\begin{theo}\label{theo:cactus}
Let $G$ be a cactus graph. 
\begin{enumerate}
\item $HG(G)=4$ if and only if $G$ contains at least two triangles.\label{two-t}
\item $HG(G)=3$ if and only if $G$ contains at least two cycles or 
a cycle of length $4$ or divisible by $3$, 
and $G$ contains fewer than two triangles.\label{one-t}
\item $HG(G)=2$ if and only if $G$ is a pseudotree with at least one edge and no
cycle of length $4$ or divisible by $3$.\label{no-t}
\end{enumerate} 
\end{theo}
The lower bound for Statement~\ref{two-t} of the theorem follows from 
Theorem~\ref{theo:glue-win} by gluing together two triangles, each with 
hatnesses $(2,4,4)$ to either end of a path (possibly of length 0) with hatnesses
$(2,4,\dots,4,2)$ at the vertices with hatness $2$
(as shown in Figure~\ref{fig:glue-cactus} on the right). Since this forms a winning subgame of 
$\langle G, \star 4\rangle$, the game $\langle G, \star 4\rangle$ is winning. 
The lower bound for Statement~\ref{one-t} in the two cycles case follows analogously by gluing
two winning cycles with hatnesses $(2,3,\dots, 3)$ to either end of a winning path 
(possibly of length 0) with hatnesses $(2,3,\dots,3,2)$ and decreasing all hatnesses in the 
resulting game to $3$ to give a winning subgame of $\langle G, \star 3\rangle$. 
The lower bound for Statement~\ref{one-t} in the one cycle case
follows from \cite[Theorem 1]{WS}. The lower bound for Statement~\ref{no-t} is trivial. 
In Section~\ref{sec:cactus}, we prove the upper bounds.

\begin{figure}
\begin{center}
\begin{tikzpicture}
 \matrix (m) [matrix of math nodes,row sep=0.4cm,column sep=0.4cm, nodes={circle,draw,thick}]
  {
       |[ultra thick]|\ &  & |[ultra thick]|\ &  &     \\
       & |[ultra thick]|\ &  &  & |[ultra thick]|\    \\ 
       & & |[ultra thick]| \ &  |[ultra thick]|\ &|[ultra thick]| \ & |[ultra thick]|\  \\
          & \ &    &  & \ &   \ &  \    \\
        & \ &  &  & \ &   \\
 };
 \path[semithick]
	(m-4-2) 	edge (m-5-2)
			edge (m-3-3)
	(m-3-6)	edge (m-4-5)
			edge (m-4-6)
			edge (m-4-7)
	(m-4-5)	edge (m-5-5);

\path[ultra thick]
	(m-1-1) 	edge (m-1-3)
			edge (m-2-2)
	(m-2-2)	edge (m-1-3)
	(m-2-5)	edge (m-3-5)
			edge (m-3-6)
	(m-3-5)	edge (m-3-6)
	(m-2-2)	edge (m-3-3)
	(m-3-4)	edge (m-3-3)
			edge (m-3-5);
\end{tikzpicture}\hspace{2.5em}\begin{tikzpicture}
 \matrix (m) [matrix of math nodes,row sep=0.4cm,column sep=0.4cm, nodes={circle,draw,thick}]
  {
       |[ultra thick]| 4 &  & |[ultra thick]| 4 &  &  &   \\
        & |[ultra thick]| 2  \\
    &   |[ultra thick]| 2 & & & & |[ultra thick]| 4   \\ 
   &    & |[ultra thick]| 4 &  |[ultra thick]| 4 & |[ultra thick]| 2 & |[ultra thick]| 2 & |[ultra thick]| 4  \\
 };
\path[ultra thick]
	(m-1-1) 	edge (m-1-3)
			edge (m-2-2)
	(m-2-2)	edge (m-1-3)
	(m-3-2)	edge (m-4-3)
	(m-4-3)	edge (m-4-4)
	(m-3-6)	edge (m-4-6)
		 	edge (m-4-7)
	(m-4-4)	edge (m-4-5)	
	(m-4-6)	edge (m-4-7);	
\path[line width=3mm]			
	(m-2-2)	edge (m-3-2)	
	(m-4-5)	edge (m-4-6);	
\path[line width=2mm, white]			
	(m-2-2)	edge (m-3-2)	
	(m-4-5)	edge (m-4-6);		
\path[line width=0.5mm]			
	(m-2-2)	edge (m-3-2)	
	(m-4-5)	edge (m-4-6);	
\end{tikzpicture}
\end{center}
\caption{A cactus graph with two triangles (left) 
and the construction that shows that the hat guessing number is at least $4$ (right).
The triple lines indicate vertices that are glued together in the construction.}
\label{fig:glue-cactus}
\end{figure}
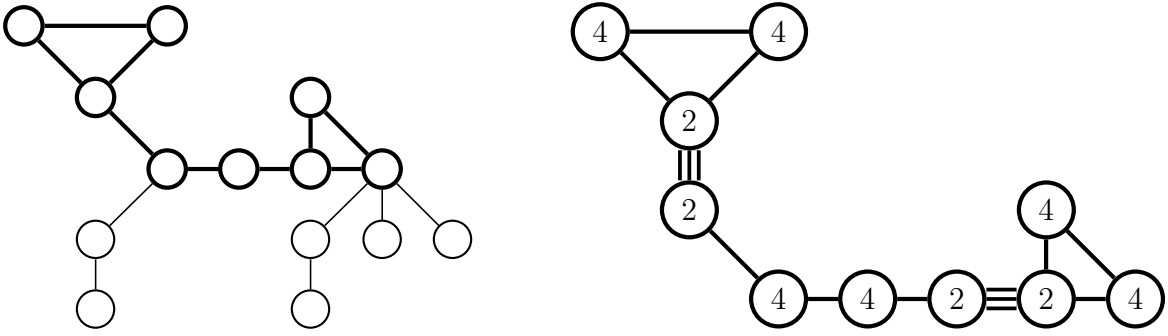

\section{Cycles}\label{sec:cycles}

Throughout this section and the next we use \emph{constructors}---rules for building 
winning or losing games from smaller games that are easier to analyze---to simplify our proofs. 
Given two graphs $G_1$ and $G_2$, with vertices $v_1$ and $v_2$, respectively, 
the graph denoted $G=G_1+_{v_1,v_2}G_2$ formed by \emph{gluing} 
$G_1$ and $G_2$ at $v_1$ and $v_2$ is given by the union of vertices and edges 
$G=G_1\cup G_2$ where $v_1$ and $v_2$ are identified as a single vertex in $G$ which is incident 
to all vertices in the union $N_{G_1}(v_1)\cup N_{G_2}(v_2)$ of their respective neighbor sets
in the original graphs. The following theorems, which concern the gluing of winning and losing graphs, 
are used as building blocks to prove our main results.

\begin{theo}[\text{\cite[Theorem 3.1]{KL21}}]\label{theo:glue-win}
Let $G_1$ and $G_2$ be graphs with $v_1\in G_1$ and $v_2\in G_2$ and let $G=G_1 +_{v_1,v_2}G_2$.
Let $h_1$ and $h_2$ be hatness functions for the graphs $G_1$ and $G_2$ respectively, and suppose
that $ \langle G_1,h_1\rangle$ and $\langle G_2,h_2 \rangle$ are both winning. 
Then $\langle G,h\rangle$ is winning where 
\[h(v)= \begin{cases} h_1(v) & \text{ if } v\in G_1\backslash\{v_1\} \\
h_2(v) & \text{ if } v\in G_2\backslash\{v_2\}\\
h_1(v_1)\cdot h_2(v_2) & \text{ if } v=v_1=v_2\\
\end{cases}\]
\end{theo}

\begin{theo}[\text{\cite[Theorem 4.1]{KLR}}]\label{theo:glue-lose}
Let $G_1$ and $G_2$ be graphs with $v_1\in G_1$ and $v_2\in G_2$ and let $G=G_1 +_{v_1,v_2}G_2$.
Let $h_1$ and $h_2$ be hatness functions for the graphs $G_1$ and $G_2$ respectively, and suppose
that $\langle G_1,h_1\rangle$ and $\langle G_2,h_2 \rangle$ are both losing with
$h_1(v_1)\geq h_2(v_2)=2$. Then $\langle G,h\rangle$ is losing where 
\[h(v)= \begin{cases} h_1(v) & \text{ if } v\in G_1 \\
h_2(v) & \text{ if } v\in G_2\backslash\{v_2\}
\end{cases}\]
\end{theo} 

With these results in mind, we are ready to move on to the proof of the reverse implications in Theorem~\ref{theo:cycles}. To prove this direction, we must show that
for a game $\cgame$ when none of the conditions hold, the game is losing. 
If none of the conditions hold and $n=3$, the result follows from \cite[Theorem 2.1]{KL21}.
Hence, we may assume $n>3$ for the rest of the proof. Further, if 
there are no vertices of hatness $2$, the result holds by \cite[Theorem 1]{WS},
so we may assume for the rest of the proof that there is at least one vertex $a$ with $h(a)=2$.
We split the proof into two parts: Condition~\ref{props} fails because Condition~\ref{seq} fails,
and Condition~\ref{props} fails because Condition~\ref{atmost4} fails.

\subsection{Proof of Theorem~\ref{theo:cycles} when Condition~\ref{seq} Fails}

\begin{proof}
We begin by assuming that Conditions~\ref{szc},~\ref{k3},~\ref{subgame}, and~\ref{seq}
all fail, and Condition~\ref{atmost4} holds. Given these assumptions, let $C$ be a cycle 
of length $n>3$ such that for every $v\in C$ we have $2\leq h(v)\leq 4$,  
and let $h(a)=2$ for some vertex $a\in C$. As a reminder, the assumptions also imply $\cgame$ 
does not contain a winning subgame and the hatness sequences $(3,2,3)$ and $(2,3,3)$ 
do not appear. Note that $a$ is the only vertex with hatness $2$ as otherwise 
Condition~\ref{subgame} would be satisfied by repeated application of 
Theorem~\ref{theo:glue-win} to the game $\langle K_2, \star 2\rangle$.

Based on the supposition, we have two possible minimal cases, as suggested in \cite{KL21}. 
It suffices to prove the claim for these minimal cases as increasing the hatnesses of 
any vertex cannot cause a losing game to become winning.
The first case is that the hatnesses of $a$'s neighbors are 
each $4$, and every other hatness is $3$. The second case is that $a$ has one neighbor
$b$ with hatness $3$ and one with hatness $4$, and the second neighbor
of $b$ has hatness $4$. Any other vertices in the cycle have hatness $3$ if they exist. 
We prove the result for these two possibilities separately, although the technique
is quite similar. 
\begin{figure}\begin{center}
\begin{tikzpicture}
  \def\n{5}
  \tkzDefPoint(0,0){O}
  \tkzDefPoint(180:2){a}
  \tkzDefPoint(108:2){b}
  \tkzDefPoint(36:2){c}
  \tkzDefPoint(324:2){d}
  \tkzDefPoint(252:2){v}

  \tkzDrawArc[delta=-0.5cm, black](O,b)(a) 
  \tkzDrawArc[delta=-0.5cm, black](O,c)(b)
  \tkzDrawArc[delta=-0.5cm, black](O,a)(v)
  \tkzDrawArc[delta=-0.5cm, black](O,v)(d)
  \tkzDrawArc[delta=-0.5cm, black](O,d)(c)
  \foreach \x [count=\i] in {v,a,b,c,\,\dots}{
    \pgfmathsetmacro{\angle}{180 - 360/\n * (\i - 2)}
    \node[circle, draw, minimum size=1cm] (\i) at (\angle:2cm) {$\x$};
  }
  \foreach \x [count=\i] in {4,2,4,3,3}{
    \pgfmathsetmacro{\angle}{180 - 360/\n * (\i - 2)}
        \node at (180-\i*360/\n+2*360/\n:3cm) {$\x$ };
  }
  
\end{tikzpicture}
\hspace{2cm}
\begin{tikzpicture}
  \def\n{5}
 \tkzDefPoint(0,0){O}
  \tkzDefPoint(180:2){a}
  \tkzDefPoint(108:2){b}
  \tkzDefPoint(36:2){c}
  \tkzDefPoint(324:2){d}
  \tkzDefPoint(252:2){v}
  \tkzDrawArc[delta=-0.5cm, black](O,b)(a) 
  \tkzDrawArc[delta=-0.5cm, black](O,c)(b)
  \tkzDrawArc[delta=-0.5cm, black](O,a)(v)
  \tkzDrawArc[delta=-0.5cm, black](O,v)(d)
  \tkzDrawArc[delta=-0.5cm, black](O,d)(c)
  \foreach \x [count=\i] in {v,a,b,c,\,\dots}{
    \pgfmathsetmacro{\angle}{180 - 360/\n * (\i - 2)}
    \node[circle, draw, minimum size=1cm] (\i) at (\angle:2cm) {$\x$};
  }
  \foreach \x [count=\i] in {4,2,3,4,3}{
    \pgfmathsetmacro{\angle}{180 - 360/\n * (\i - 2)}
        \node at (180-\i*360/\n+2*360/\n:3cm) {$\x$ };
  }
\end{tikzpicture}\vspace{-1em}
\end{center}
\caption{The two cases described in the proof of Theorem~\ref{theo:cycles}, 
where the ellipsis indicates an arbitrary (possibly 0) number of vertices with 
hatness $3$.}
\label{fig:cycles}
\end{figure}
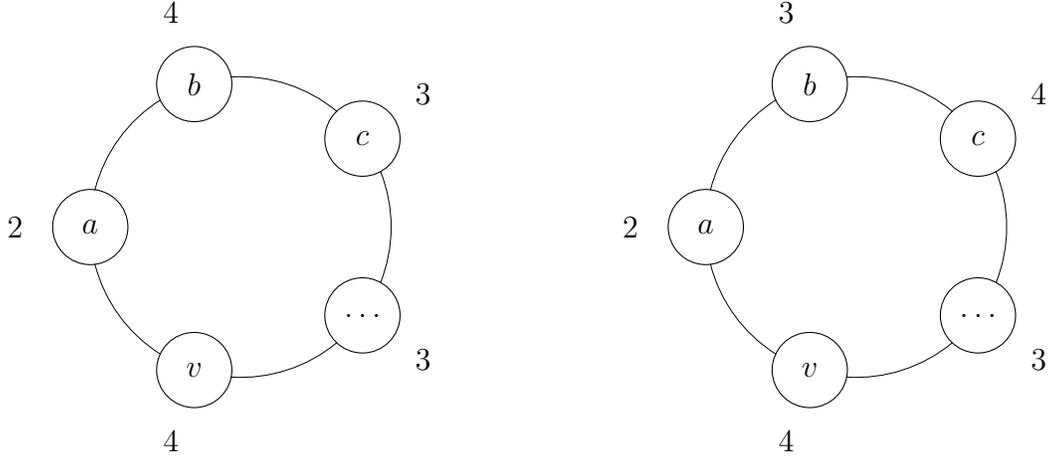

In the first case, let $v,a,b,c$ be a sequence of vertices in the cycle with 
\[h(v)=4,\,h(a)=2,\,h(b)=4,\,h(c)=3\] 
and any other vertices have hatness 3,
as shown on the left side of Figure~\ref{fig:cycles}.
\begin{figure}
\begin{center}
\begin{tikzpicture}
  \def\n{5}
  \tkzDefPoint(0,0){O}
  \tkzDefPoint(180:2){a}
  \tkzDefPoint(108:2){b}
  \tkzDefPoint(36:2){c}
  \tkzDefPoint(324:2){d}
  \tkzDefPoint(252:2){v}

  \tkzDrawArc[delta=-0.5cm, black, dashed](O,b)(a) 
  \tkzDrawArc[delta=-0.5cm, black, dashed](O,c)(b)
  \tkzDrawArc[delta=-0.5cm, black, dashed](O,a)(v)
  \tkzDrawArc[delta=-0.5cm, black](O,v)(d)
  \tkzDrawArc[delta=-0.5cm, black](O,d)(c)
  \foreach \x [count=\i] in {c,\,\dots,v}{
    \pgfmathsetmacro{\angle}{180 - 360/\n * (\i - 4)}
    \node[circle, draw, minimum size=1cm] (\i) at (\angle:2cm) {$\x$};
    }
  \foreach \x [count=\i] in {a,b}{
    \pgfmathsetmacro{\angle}{180 - 360/\n * (\i - 1)}
    \node[circle, draw, minimum size=1cm, dashed] (\i) at (\angle:2cm) {$\x$};
  }
  \foreach \x [count=\i] in {3,,,2,3}{
    \pgfmathsetmacro{\angle}{180 - 360/\n * (\i - 2)}
        \node at (180-\i*360/\n+2*360/\n:3cm) {$\x$ };
  }
\end{tikzpicture}
\hspace{2cm}
\begin{tikzpicture}
  \def\n{5}
  \tkzDefPoint(0,0){O}
  \tkzDefPoint(180:2){a}
  \tkzDefPoint(108:2){b}
  \tkzDefPoint(36:2){c}
  \tkzDefPoint(324:2){d}
  \tkzDefPoint(252:2){v}
  \tikzset{compass style/.append style={black, dashed}}
  \tkzDrawArc[delta=-0.5cm, black, dashed](O,b)(a) 
  \tkzDrawArc[delta=-0.5cm, black, dashed](O,c)(b)
  \tkzDrawArc[delta=-0.5cm, black, dashed](O,a)(v)
  \tikzset{compass style/.append style={solid}}
  \tkzDrawArc[delta=-0.5cm, black](O,v)(d)
  \tkzDrawArc[delta=-0.5cm, black](O,d)(c)
  \foreach \x [count=\i] in {c,\,\dots,v}{
    \pgfmathsetmacro{\angle}{180 - 360/\n * (\i - 4)}
    \node[circle, draw, minimum size=1cm] (\i) at (\angle:2cm) {$\x$};
    }
  \foreach \x [count=\i] in {a,b}{
    \pgfmathsetmacro{\angle}{180 - 360/\n * (\i - 1)}
    \node[circle, draw, minimum size=1cm, dashed] (\i) at (\angle:2cm) {$\x$};
  }
  \foreach \x [count=\i] in {2,,,3,3}{
    \pgfmathsetmacro{\angle}{180 - 360/\n * (\i - 2)}
        \node at (180-\i*360/\n+2*360/\n:3cm) {$\x$ };
  }
\end{tikzpicture} \vspace{-1em}
\end{center}
\caption{The two cases described in the proof of Theorem~\ref{theo:cycles}, 
after coloring $a$ and $b$ and committing to a subset of colors for $c$ and $v$.
The ellipsis indicates an arbitrary (possibly 0) number of vertices with 
hatness $3$.}
\label{fig:cycles-post}
\end{figure}
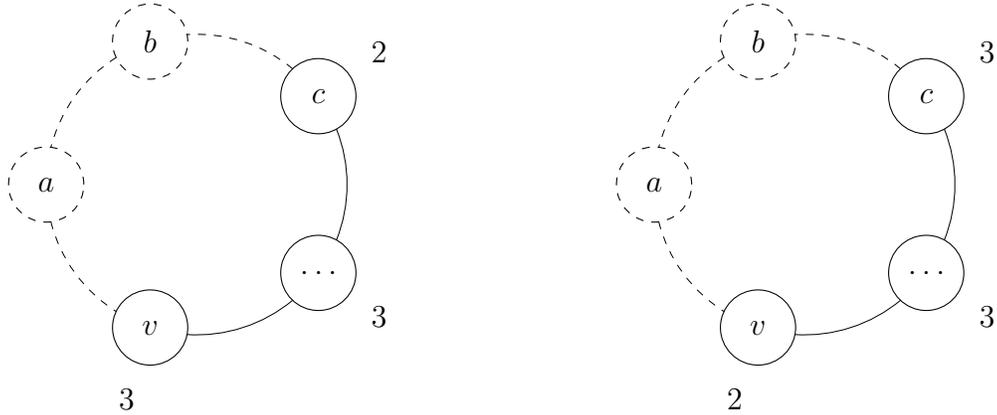
Let $x_b$ be the color guessed 
least frequently by vertex $b$ (breaking ties arbitrarily). By the pigeonhole principle, 
$b$ can only guess $x_b$ at most once. Let $x_a$ and $x_c$ be the colors
of $a$ and $c$, respectively, when $b$ guesses $x_b$ (if $x_b$ is never guessed,
$a$ and $c$ can be chosen arbitrarily). Let $X_v$ 
be the set of colors for $v$ so that when $b$ is colored
$x_b$, and $v$ is colored from $X_v$, $a$ guesses color $x_a$. 
Suppose toward a contradiction that $|X_v|\geq 2$. Then we commit to coloring $a$ with the 
color which is not $x_a$, $b$ with the color $x_b$, $c$ with any of the three colors,
and $v$ with a color from $X_v$, so $a$ and $b$ guess incorrectly. Hence,
$\cgame$ is winning only if there exists a winning strategy for
the game on a path of length at least $2$ where every vertex 
has hatness $3$ except for one vertex with hatness $2$ as depicted in the right 
side of Figure~\ref{fig:cycles-post}. However, Theorem~\ref{theo:glue-lose}
implies that this path is losing. Hence, in any winning strategy,
we may assume that $|X_v|\leq 1$. Thus, we can commit to coloring $a$ 
with $x_a$, $v$ with one of at least $3$ colors not in $X_v$, coloring $b$ 
with $x_b$ and $c$ with one of the two colors not equal to $x_c$. 
This ensures again that $a$ and $b$ each guess incorrectly, 
and we again reduce to the path where every vertex has hatness $3$ 
except one with hatness $2$ (in this case depicted on the left side of Figure~\ref{fig:cycles-post}), 
which is losing. Hence, $C$ is losing. 

Next, we prove the result in the second case, where we have
a sequence of vertices $v,a,b,c$ in the cycle with 
\[h(v)=4,\,h(a)=2,\,h(b)=3,\,h(c)=4,\] and any other vertices in the cycle have hatness 
$3$ if they exist (as shown on the right side of Figure~\ref{fig:cycles}).
 As before, we let $x_b$ be the least frequently guessed 
color for $b$, breaking ties arbitrarily. By the pigeonhole principle, $x_b$ is guessed 
at most twice. Let $(x_a,x_c)$ and $(y_a,y_c)$ be the colorings of $(a,c)$
for which $b$ guesses $x_b$. If $x_b$ is guessed at most once
then we can choose arbitrary values for any undetermined pairs. Now, there are 
three possibilities to consider. The first case is $x_c=y_c$; the second is $x_a=y_a$; 
the third is $x_a\neq y_a$ and $x_c\neq y_c$. 

In the first case, we commit to coloring $b$ with the color $x_b$, 
and $c$ with a color other than $x_c$. This reduces 
to a path where the first vertex has hatness $2$, the second vertex has 
hatness $4$, and the remaining vertices have hatness $3$, which is 
losing by Theorem~\ref{theo:glue-lose}. 

In the second case, we proceed as in the $(4,2,4,3)$ case. 
Let $X_v$ be the set of colors for $v$ such that
$a$ guesses $x_a$. As before, if $|X_v|\geq 2$ we color $a$ 
with the color not equal to $x_a$, $v$ from $X_v$ and $c$ arbitrarily; 
otherwise, we color $a$ with $x_a$, $v$ from colors not in $X_v$ and $c$ 
from colors not $x_c$ or $y_c$. As before, these choices ensure that $a$ and $b$
guess incorrectly and leave losing paths. (Shown on the right and left sides 
of Figure~\ref{fig:cycles-post}, respectively. We actually
have $4$ possible colors for $c$ in the first case but only $3$ are needed.) 

Finally, in the third case, assume without loss of generality that 
$a$ guesses $x_a$ at most as frequently as $y_a$ when $b$ is 
colored $x_b$. Then we can commit to coloring $b$ with the color
$x_b$, vertex $a$ with the color $x_a$, and $c$ with a color
that is not $x_c$, of which there are $3$ choices. 
Since $x_a$ is guessed at most as frequently as $y_a$, there are at least 
$2$ colors for $v$ such that $a$ does not guess
$x_a$, and  we can again reduce the game to a losing path (in this case,
the one shown on the right in Figure~\ref{fig:cycles-post}). Thus,
$\cgame$ is losing in all cases when Conditions~\ref{szc},~\ref{k3},~\ref{subgame},~and~\ref{seq}
fail.\end{proof}
It remains to show that $\cgame$ is losing when 
Conditions~\ref{szc},~\ref{k3},~\ref{subgame},~and~\ref{atmost4} fail. 

\subsection{Proof of Theorem~\ref{theo:cycles} when Condition~\ref{atmost4} Fails}

We begin with lemmas to simplify the argument for this case.
Our first lemma allows us to delete two vertices of a cycle,
leaving behind a path that is winning if the original cycle
is winning. An application of the lemma is depicted in Figure~\ref{fig:d2vlem}.

\begin{lemma}\label{lem:delete}
Let $C=(V,E)$ be a cycle of length at least $4$, 
let $t,u,v,w\in V$ be consecutive vertices in the cycle so that $(t,u),(u,v),(v,w)\in E$,
and let $h(\cdot)$ be a hatness function with $h(v)>h(u)$.
We define 
\[h'(x)=\begin{cases}
\left\lceil h(w) \left(1-\left\lceil\frac{h(v)}{h(u)}\right\rceil^{-1}\right)\right\rceil & \text{if } x=w\\
h(t)-\left\lfloor \frac{h(t)}{h(u)}\right\rfloor & \text{if }  x = t\\
h(x) &  \text{if } x\in V\backslash \{t,u,v,w\}
\end{cases}\]
If $\cgame$ is winning then $\langle C\backslash \{u,v\}, h'\rangle$ is winning.
\end{lemma}

\begin{figure}
\begin{center}
\begin{tikzpicture}
  \def\n{6}
  \foreach \x [count=\i] in {t,...,y}{
    \pgfmathsetmacro{\angle}{180 -360/\n * (\i - 1)}
    \node[circle, draw, minimum size=1cm] (\i) at (90+\angle:2cm) {$\x$};
        \node at (180-\i*360/\n+2.5*360/\n:3cm) {$a_\x$ };
  }
 
          \node at (180-4*360/\n+2.5*360/\n:3cm) {\phantom{$\left\lceil a_w \left(1-\left\lceil\frac{a_v}{a_u}\right\rceil^{-1}\right)\right\rceil$}};
          \node at (180-1*360/\n+2.5*360/\n:3cm) {\phantom{$a_t-\left\lfloor \frac{a_t}{a_u}\right\rfloor$}};

  \tkzDefPoint(0,0){O}
  \tkzDefPoint(30:2){t}
  \tkzDefPoint(90:2){u}
  \tkzDefPoint(150:2){v}
  \tkzDefPoint(210:2){w}
  \tkzDefPoint(270:2){x}
  \tkzDefPoint(330:2){y}
  \tkzDrawArc[delta=-0.5cm, black](O,t)(u) 
  \tkzDrawArc[delta=-0.5cm, black](O,u)(v)
  \tkzDrawArc[delta=-0.5cm, black](O,v)(w)
  \tkzDrawArc[delta=-0.5cm, black](O,w)(x)
  \tkzDrawArc[delta=-0.5cm, black](O,x)(y)
  \tkzDrawArc[delta=-0.5cm, black](O,y)(t)
\end{tikzpicture}
\hspace{1.5cm}
  \begin{tikzpicture}
  \def\n{6}
  \foreach \x [count=\i] in {w,x,y,t}{
    \pgfmathsetmacro{\angle}{180 -360/\n * (\i - 1)}
    \node[circle, draw, minimum size=1cm] (\i) at (-90+\angle:2cm) {$\x$};
  }
  \foreach \x [count=\i] in {u,v}{
    \pgfmathsetmacro{\angle}{180 -360/\n * (\i - 1)}
    \node[circle, draw, dashed,minimum size=1cm] (\i) at (30+\angle:2cm) {$\x$};
  }
  
  \tkzDefPoint(0,0){O}
  \tkzDefPoint(30:2){t}
  \tkzDefPoint(90:2){u}
  \tkzDefPoint(150:2){v}
  \tkzDefPoint(210:2){w}
  \tkzDefPoint(270:2){x}
  \tkzDefPoint(330:2){y}

  \tkzDrawArc[delta=-0.5cm,dashed, black](O,u)(v)
  \tkzDrawArc[delta=-0.5cm,dashed, black](O,v)(w)
  \tkzDrawArc[delta=-0.5cm, dashed, black](O,w)(x)

  \tikzset{compass style/.append style={black}}
  \tkzDrawArc[delta=-0.5cm, black](O,t)(u) 
  \tkzDrawArc[delta=-0.5cm, black](O,x)(y)
  \tkzDrawArc[delta=-0.5cm, black](O,y)(t)

  \foreach \x [count=\i] in {t,y,x,w}{
  }
  \def\hatnesses{a_t-\left\lfloor \frac{a_t}{a_u}\right\rfloor ,a_y,a_x,
         \left\lceil a_w \left(1-\left\lceil\frac{a_v}{a_u}\right\rceil^{-1}\right)\right\rceil}
   
   \foreach \x [count=\i] in \hatnesses {
            \node at (180+\i*360/\n+.5*360/\n:3cm) {$\x$ };
  }

\end{tikzpicture}\vspace{-2em}
\end{center}
\caption{A cycle before and after applying Lemma~\ref{lem:delete}}
\label{fig:d2vlem}
\end{figure}
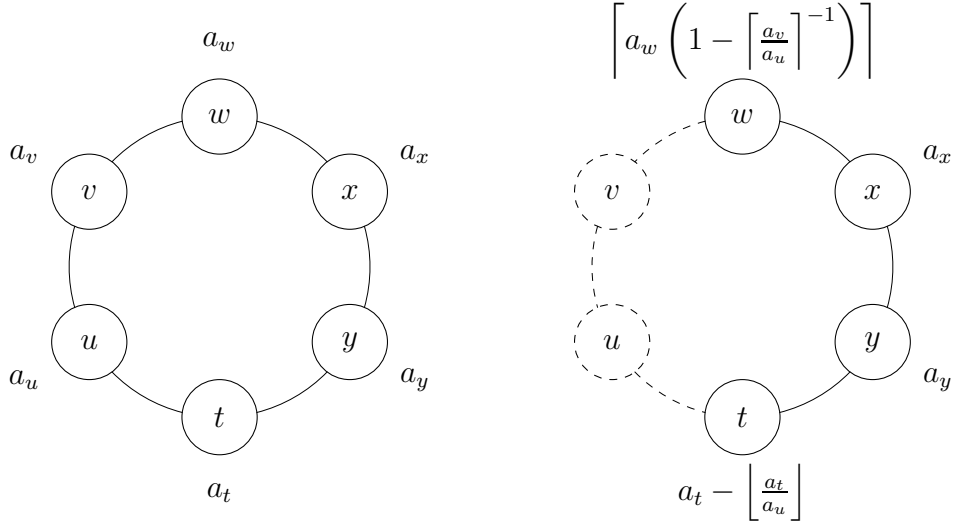

\begin{proof}
The first step is to show that there exists a color $x_v\in [h(v)]$ such that
for each color $x_u\in [h(u)]$ there is a set $X_w(x_u,x_v)\subseteq [h(w)]$ of colors for $w$ with
\begin{equation}|X_w(x_u,x_v)| \geq 
\left\lceil h(w) \left(1-\left\lceil\frac{h(v)}{h(u)}\right\rceil^{-1}\right)\right\rceil,\label{eqn:Xw}\end{equation}
so that if $v$ is colored $x_v$,  
$u$ is colored $x_u$, and $w$ is colored $x_w\in X_w(x_u,x_v)$ then $v$ guesses
incorrectly. Indeed, suppose toward a contradiction that no such $x_v$ exists. Then,
for each color $x_v\in[h(v)]$, there exists a color $x_u\in[h(u)]$ that violates Equation~\ref{eqn:Xw}.
That is, there exists a set $Y_w(x_u,x_v)\subseteq [h(w)]$ of colors for $w$ such that 
$v$ correctly guesses $x_v$ whenever $w$ is colored from $Y_w(x_u,x_v)$ and $u$ is colored $x_u$, and
\[|Y_w(x_u,x_v)|> h(w)- \left\lceil h(w) \left(1-\left\lceil\frac{h(v)}{h(u)}\right\rceil^{-1}\right)\right\rceil.\]
By the pigeonhole principle, there is some $y_u\in[h(u)]$ which is used to generate $Y_w(y_u,x_v)$
for a subset $Y_v\subseteq[h(v)]$ of colors with $|Y_v|\geq \left\lceil\frac{h(v)}{h(u)}\right\rceil$. 
Summing over this set, we get
\begin{align*}\sum_{y_v\in Y_v}|Y_w(y_u,y_v)|
&> \left\lceil\frac{h(v)}{h(u)}\right\rceil 
\left(h(w)- \left\lceil h(w) \left(1-\left\lceil\frac{h(v)}{h(u)}\right\rceil^{-1}\right)\right\rceil\right)\\
&\geq \left\lceil\frac{h(v)}{h(u)}\right\rceil \left(h(w)\left\lceil\frac{h(v)}{h(u)}\right\rceil^{-1}\right)\\
&=h(w).
\end{align*}
This is a contradiction since for a fixed color $y_u$, there are exactly $h(w)$ guesses made
by $v$. Hence, there exists some $z_v\in[h(v)]$ with the desired property. 

Fix the color of $v$ to be $z_v$. With $z_v$ fixed, fix the color of $u$ to be the least frequently 
guessed color of $u$, denoted $z_u$. By our choice of $z_u$, there are at least 
$h(t)-\left\lfloor \frac{h(t)}{h(u)}\right\rfloor$ colors remaining for $t$, all of which ensure that 
$u$ guesses incorrectly. Finally, there are 
$\left\lceil h(w) \left(1-\left\lceil\frac{h(v)}{h(u)}\right\rceil^{-1}\right)\right\rceil$ colors
for $w$ that ensure $v$ guesses incorrectly. Hence, after the vertices of $C$ commit
to a guessing strategy, we can commit to coloring the vertices of $C\backslash \{u,v\}$
using $h'(x)$ colors for each $x\in C\backslash \{u,v\}$ in a way that guarantees
$u$ and $v$ guess incorrectly. Hence, if $\cgame$ is winning 
then the corresponding winning strategy must give a winning strategy on 
$\langle C\backslash \{u,v\}, h'\rangle$. 
\end{proof}

Our next lemma handles a step present in a few cases in Theorem~\ref{theo:cycles}.
Equivalent results have been shown in other works, but we reprove the result here in terms that
are useful for our arguments. The lemma can be thought of as a ``vertex deletion'' lemma as it 
allows us to delete a vertex with hatness $5$ lying on a path without changing the winningness of the
game. 
\begin{lemma}\label{hatness-5-path}
Let $P$ be a path and let $h$ be a hatness function with $h(v)=5$ for some $v\in P$.  
The game $\langle P,h\rangle$ is winning if and only if there is a connected 
proper subgraph of $P$ which is winning. 
\end{lemma}

\begin{proof}
Let $P$ be a path and let $h$ be a hatness function with $h(v)=5$ for some $v\in P$.  
If $P$ contains a winning connected proper subgraph then $P$ is winning.
Suppose $P$ does not contain a winning connected proper subgraph and let $L, R$ 
be the resulting disjoint paths when $v$ is deleted from $P$. 
Then $L$ and $R$ are connected proper subgraphs of $P$, so they are losing
with their given hatness functions. The three vertex path with 
hatnesses $(2,5,2)$ is also losing. By Theorem~\ref{theo:glue-lose}, we can glue together
this path with $L$ on one endpoint and $R$ on the other endpoint 
to get a new losing path where the hatnesses are inherited from the
endpoints of $L$ and $R$. Note that this is exactly the game $\langle P,h\rangle$, 
so $P$ is losing. \end{proof}

Now, we are ready to prove Theorem~\ref{theo:cycles} when 
Conditions~\ref{szc},~\ref{k3},~\ref{subgame}, and~\ref{atmost4}
all fail. In particular, when Condition~\ref{atmost4} fails, there is a vertex
with hatness at least $5$. 

\begin{proof}
Let $C$ be a cycle of length at least $4$ with at least one vertex of hatness $5$ and
no winning proper subgraph. Observe that
if there are no vertices with hatness $2$,
the game is losing by \cite[Theorem 1]{WS}.
Hence, we may assume there is at least one vertex with hatness $2$. 
We may also assume without loss of generality that every vertex has hatness $2$, $3$, or $5$ 
since decreasing the hatness of a vertex does not make a winning game losing. 
Finally, on any path between two vertices of hatness $2$, 
we assume there is at most one vertex with hatness $5$ for the same reason. 

If there is a vertex with hatness $2$ adjacent to a vertex of hatness $5$
then there is one of the following sequences of hatnesses: $(2,5,2)$, 
$(5,2,3,5)$, $(5,2,3,3)$. Otherwise, there is no vertex with hatness
$2$ next to a vertex of hatness $5$, and we must have a sequence 
$(3,2,3,3)$ or $(3,2,3,5)$ in the cycle. It is easy to see that these cover 
all possible cases, with the exception of the exclusion of the case $(5,2,5,3)$. 
This can be attributed to the fact that the presence of two vertices
of hatness $5$ indicates the presence of a vertex of hatness $2$ somewhere
else in the cycle, with at most one vertex of hatness $5$ as a neighbor.  
Now, we handle each case. 

For the sequence $(2,5,2)$, we can choose a color for the vertex of hatness $5$ 
that is never guessed, leaving a path that induces proper subgame of 
$\cgame$. Hence, this path must be losing, and the result holds.

For the sequence $(5,2,3,5)$, we apply Lemma~\ref{lem:delete} with
\[h(t)=5,\, h(u)=2,\, h(v)=3,\, h(w)=5\]
to get a path with endpoints $t$ and $w$ with $h'(t)=3$ and $h'(w) = 3$. 
Observe that $t$ and $w$ have degree one in the new graph, and so they are deletable 
by Theorem~\ref{theo:glue-lose}. This leaves a proper subgame that must must be losing.
Hence, $\cgame$ is losing. 

For the sequence $(5,2,3,3)$, we again apply Lemma~\ref{lem:delete} with 
\[h(t)=5,\, h(u)=2,\, h(v)=3,\, h(w)=3\]
to get a path with endpoints $t$ and $w$ with $h'(t)=3$ and $h'(w) = 2$.
Observe that $t$ has degree one in the new graph, and so it is deletable. 
Suppose toward a contradiction that the new path is winning. Then, by
Theorem~\ref{theo:glue-win} we can glue one end of a path $(2,3,2)$ 
which is winning to $w$ and create a new graph that is still winning but where 
$w$ has hatness $4$. Then, we can decrease the hatness of $w$ to 
$3$, and the graph is still winning. However, this new graph is a proper subgraph 
of the original cycle $C$, which yields a contradiction. Hence, $\cgame$ is losing. 

For the sequence $(3,2,3,3)$,  we apply Lemma~\ref{lem:delete} with
\[h(t)=3,\, h(u)=2,\, h(v)=3,\, h(w)=3\]
to get $h'(t)=2$ and $h'(w) = 2$. Suppose toward a contradiction that 
the resulting path is winning. Then we can glue a copy of the winning
path with hatnesses $(2,2)$ to $t$ and $(2,3,2)$ to $w$, which gives a new 
winning path where $t$ and $w$ both have hatness $4$ by 
Theorem~\ref{theo:glue-win}. Decreasing
the hatnesses for $t$ and $w$ to $3$ preserves the winningness of the path.
Note that this path contains a vertex of hatness $5$. By 
Lemma~\ref{hatness-5-path}, there must be a winning connected
proper subgraph. However, any connected proper subgraph 
of this path cannot contain both endpoints and  
is thus also a proper subgraph of the original cycle,
yielding a contradiction. Hence, $\cgame$ is losing. 

For the sequence $(3,2,3,5)$, we apply Lemma~\ref{lem:delete} with 
\[h(t)=3,\, h(u)=2,\, h(v)=3,\, h(w)=5\]
to get $h'(t)=2$ and $h'(w) = 3$. Since $w$ has hatness $3$ in the new graph,
$w$ is deletable. After deleting $w$, suppose toward a contradiction that the graph is 
winning. Then we can glue an edge with hatnesses $(2,2)$ to $t$, 
creating a winning graph where $t$ has hatness $4$. We can
then decrease the hatness of $t$ to $3$. This graph is 
now a proper subgame of $\cgame$, which must be losing,
a contradiction. Hence, $\cgame$ is losing.
\end{proof}

\section{Cactus Graphs}\label{sec:cactus}

We start by proving the general upper bound that every
cactus graph has hat guessing number at most $4$, which
proves Theorem~\ref{theo:cactus}, Statement~\ref{two-t}. 
Then, we show how the same argument can be modified
to prove Statement~\ref{one-t}. Finally, we observe that
Statement~\ref{no-t} follows immediately from~\cite{KL18}.

\begin{defi} Let $G$ be a cactus graph with a cycle $C$. We 
say that $C$ is a \emph{leaf cycle} if deleting the edges of $C$ 
leaves at most $1$ connected component with a cycle (so all
other connected components are acyclic).
\end{defi}
Figure~\ref{fig:cactus} depicts the leaf cycles of a cactus graph in bold. 

\begin{figure}
\begin{center}
 \begin{tikzpicture}
  \matrix (m) [matrix of math nodes,row sep=0.4cm,column sep=0.4cm, nodes={circle,draw,thick}]
  {
      \ & \ & |[ultra thick]| \ & |[ultra thick]| \  &  & \ & \ &  & |[ultra thick]|\ &  & |[ultra thick]|\ &  &  &   \\
        & \ & |[ultra thick]| \ &    & |[ultra thick]| \ &  &  & \ &  & |[ultra thick]|\ &  &  & |[ultra thick]|\ &   \\ 
      \  & \ &  & |[ultra thick]|  \  &  & \ & \ &  & \ & \ & \ &|[ultra thick]| \ &  & |[ultra thick]|\  \\
        & \ &  &    &  & \ &  & \ &  & \ &  & |[ultra thick]|\ &  & |[ultra thick]|\   \\
        &  &  &    & \ &  &  &  & \ &  &  &  & |[ultra thick]|\ &   \\
 };
  \path[semithick]
	(m-1-1) 	edge (m-1-2)
	(m-1-2) 	edge (m-1-3)
			edge (m-2-2)
	(m-2-2) 	edge (m-3-2)
	(m-2-5)	edge (m-1-6)
			edge (m-3-6)
	(m-3-2) 	edge (m-4-2) 
	(m-3-1) 	edge (m-3-2)
	(m-1-6) 	edge (m-1-7)
	(m-1-7) 	edge (m-2-8)
	(m-2-8) 	edge (m-3-7)
	(m-3-6) 	edge (m-3-7)
			edge (m-4-6)
	(m-4-6) 	edge (m-5-5)
	(m-3-7) 	edge (m-4-8)
	(m-4-8) 	edge (m-3-9)
			edge (m-5-9)
	(m-5-9)	edge (m-4-10)
	(m-4-10)	edge (m-3-9)
			edge (m-3-10)
	(m-3-10) 	edge (m-3-11)
			edge (m-2-10)
	(m-3-11) 	edge (m-3-12);
\path[ultra thick]
	(m-2-5) 	edge (m-3-4)
	(m-2-3) 	edge (m-3-4)
	(m-1-3) 	edge (m-2-3)
			edge (m-1-4)
	(m-1-4) 	edge (m-2-5)
	(m-3-12) 	edge (m-2-13)
			edge (m-4-12)
	(m-4-12) 	edge (m-5-13)
	(m-2-10)	edge (m-1-11)
			edge (m-1-9)
	(m-1-9)	edge (m-1-11)
	(m-5-13) 	edge (m-4-14)
	(m-3-14) 	edge (m-4-14)
			edge (m-2-13);
\end{tikzpicture}\vspace{-1em}
\end{center}
\caption{A cactus graph with the leaf cycles in bold.}
\label{fig:cactus}
\end{figure}
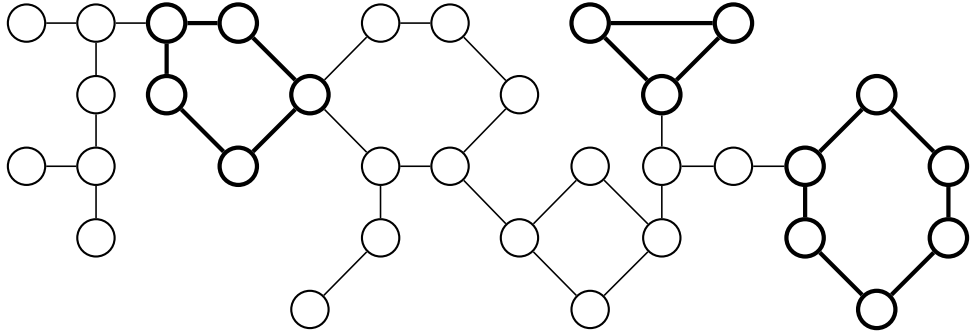

\begin{proof}[Proof of Theorem~\ref{theo:cactus}]
First, we prove the general upper bound of $HG(G)\leq 4$ for all cactus graphs $G$ by 
induction on the number of vertices. The base case is trivial. 
Suppose that the result holds for all cactus graphs on $n\geq 1$ vertices, and 
let $G$ be a cactus graph with $n+1$ vertices. We show that 
$HG(G)<5$. If $G$ has a vertex $v$ with degree $1$ then 
\cite[Theorem~1.8]{ABST} shows that $HG(G)=HG(G\backslash\{v\})$, and the 
result holds by induction. Thus, we may assume that $G$ has no
vertices of degree $1$. If $G$ has no cycles then $G$ is a tree, and 
$HG(G)=2$. Hence, we may also assume $G$ has at least one cycle.
In particular, this implies $G$ has a non-empty leaf cycle $C$. We can decompose 
$G$ as $G=C+_{v,v}G'$ at a vertex $v\in C\cap G'$,
where $G'$ is a cactus graph on fewer than $n+1$ vertices.
By the induction hypothesis, $HG(G')< 5$, so the game $\langle G',\star 5\rangle$ is losing.
Theorem~\ref{theo:cycles} shows that $\cgame$ is losing
where $h(v)=2$ and $h(w)=5$ for all $w\neq v$. Then $\langle G, \star 5\rangle$ is losing by
Theorem~\ref{theo:glue-lose}, so $HG(G)<5$ and the result holds.

The same argument can easily be modified to prove the upper bound for
Statement~\ref{one-t}, by assuming $HG(G)<4$ in the induction hypothesis
for graphs satisfying Statement~\ref{one-t}. If the graph has exactly one cycle,
then it is a pseudotree and the result holds by \cite{KL18}. Otherwise, $G$
has at least two cycles and thus, at least two leaf cycles.
Statement~\ref{one-t} ensures at most one triangle, so we can choose a non-triangle
leaf cycle $C$ in $G$ and write $G=C+_{v}G'$, where $G'$ is a cactus graph on 
fewer than $n+1$ vertices. Then $\cgame$ is losing where $h(v)=2$ and $h(w)=4$
for all $w\neq v$, and $\langle G',\star 4\rangle$ is losing by the induction hypothesis,
so $\langle G,\star 4\rangle$ is losing by Theorem~\ref{theo:glue-lose}. 

As Statement~\ref{no-t} follows from \cite{KL18}, the proof of 
Theorem~\ref{theo:cactus} is complete.
\end{proof}

\section{Open Problems}

Theorem~\ref{theo:cycles} (along with some lemmas about leaf deletion) 
can extend the result for pseudotrees \cite{KL18} to arbitrary hatness 
functions~\cite{IMJ}. Naturally, we would like to do so for graphs with 
more than one cycle. The treelike structure of cactus graphs, which this paper 
leveraged, could help make them the next step. One challenge that must be 
overcome is that the hatness for any individual vertex of a winning cactus
graph may be arbitrarily large. For example, we can take $n$ copies of the 
triangle with hatnesses $(2,4,4)$ and glue them together at the vertices of 
hatness $2$. The resulting graph is a winning cactus with a vertex of hatness $2^n$. 
Even worse, we can glue multiple copies of this graph together to
get arbitrarily many vertices with hatness $2^n$. 

Instead, we could turn to \emph{Theta graphs}, consisting of three internally vertex 
disjoint paths that share common endpoints. Thetas are the ``simplest'' non-cactus graphs, 
and yet we lack even an analogue to Theorem~\ref{theo:cactus}. Every theta has maximum
degree $3$, so the hat guessing number of a theta is trivially at most $8$
by \cite{MF}. Careful analysis using a generalized 
version of Lemma~\ref{lem:delete} shows that thetas have hat 
guessing number at most $4$. We postpone a full characterization 
of thetas for both constant and general hatness functions to a future work.  \\

\noindent {\bf Acknowledgment:}
We thank Noga Alon for connecting the authors and Stephen Melczer for helpful suggestions.

\pagebreak

\end{document}